\documentclass[12pt]{amsart}
\usepackage{amsmath,amssymb,latexsym}
\usepackage{fullpage}
\usepackage{amscd}
\theoremstyle{plain}
\newtheorem{theorem}{Theorem}
\newtheorem{proposition}{Proposition}

\newtheorem{conjecture}{Conjecture}

\newtheorem{problem}{Problem}

\begin{document}

\title{Non-Generic Unramified Representations in Metaplectic 
Covering Groups }
\author{David Ginzburg}

\begin{abstract}
Let $G^{(r)}$ denote the metaplectic covering group of the linear
algebraic group $G$.
In this paper we study  conditions on unramified representations of the group $G^{(r)}$ not to have a nonzero Whittaker function. We state a general Conjecture about the possible unramified characters $\chi$ such that the unramified sub-representation of  $Ind_{B^{(r)}}^{G^{(r)}}\chi\delta_B^{1/2}$ will have no nonzero Whittaker function. We prove this Conjecture for the groups $GL_n^{(r)}$ with $r\ge n-1$,
and for the exceptional groups $G_2^{(r)}$ when $r\ne 2$. 
\end{abstract}

\thanks{ The author is partly supported by the Israel Science
Foundation grant number  259/14}

\address{ School of Mathematical Sciences\\
Sackler Faculty of Exact Sciences\\ Tel-Aviv University, Israel
69978 } \maketitle \baselineskip=18pt

\section{introduction}\label{intro}
Let $F$ denote a local field and let $G$ denote a split linear algebraic group defined over $F$. Let $B$ denote the Borel subgroup of $G$ and let $\chi$ denote an unramified character of $B$.
Let $f_\chi$ denote the unramified vector in the representation $Ind_B^G\chi\delta_B^{1/2}$.  Let
\begin{equation}\label{whit1}
W_\chi(g)=\int_U f_\chi(w_0ug)\psi_U(u)du
\end{equation}
denote the Whittaker function associated with $f_\chi$. Then this integral converges at some domain of $\chi$
and admits an analytic continuation to all $\chi$. 

Assume that for some unramified character $\chi_0$, the function 
$W_{\chi_0}(g)$ is zero for all $g$. This can happen if and only if 
$W_{\chi_0}(e)$ is zero.
It follows from the Casselman Shalika \cite{C-S} formula that there 
exists a parabolic subgroup $P$ of $G$, containing $B$, but not equal to it, such that $f_{\chi_0}$ is in the induced representation
$Ind_P^G\mu_0\delta_P^{1/2}$. Here $\mu_0$ is an unramified character of $P$. Conversely,  it is easy to prove that for such an induced  representation, integral  \eqref{whit1} is zero.

In this paper we consider a similar situation in the metaplectic covering groups. In contrast to the linear group case, representations of metaplectic covering groups tend not to have a unique Whittaker function. In details, let $\eta\in T_0^{(r)}\backslash T^{(r)}$ where $T^{(r)}$ is the inverse image in $G^{(r)}$ of the maximal torus $T$
of $G$, and $T_0^{(r)}$ a maximal abelian subgroup in $T^{(r)}$. See
Section \ref{basic}. Given an unramified character of $T$, we can form the induced representation $Ind_{B^{(r)}}^{G^{(r)}}\chi\delta_B^{1/2}$. Then we can form a Whittaker function defined on this induced space by all functions of the form
\begin{equation}\label{whit2}
W_{\eta}^{(r)}(g)=\int_{U} f^{(r)}(\eta w_0ug)\psi_{U}(u)du
\end{equation}
where $f^{(r)}$ denotes a vector in this induced space. Let $\pi$ denote the sub-representation of $Ind_{B^{(r)}}^{G^{(r)}}\chi\delta_B^{1/2}$ generated by the unramified vector which we shall denote by $f^{(r)}_\chi$. We will say that $\pi$ is generic if for some $\eta$ integral $W_{\eta}^{(r)}(g)$ is not zero when 
$f^{(r)}=f^{(r)}_\chi$.  Otherwise we say that $\pi$ is not generic. 
In this paper we study the following
\begin{problem}\label{prob1}
With the above notations, suppose that $\pi$ is not generic. What
can be said about the character $\chi$?
\end{problem}

A way to construct examples of non-generic representations is  using Theta representations.
Let $H^{(r)}(F)$ denote the $r$ metaplectic covering group of a linear algebraic group $H$.  
Let $\Theta_H^{(r)}$ denote a Theta representation defined on 
$H^{(r)}$. This representation can be realized as the unramified 
sub-representation of a certain induced representation $Ind_{B_H^{(r)}}^{H^{(r)}}\chi_\Theta^H\delta_{B_H}^{1/2}$. For example, when $H=GL_n$,
this representation was defined in \cite{K-P}. For this group $\chi_\Theta^H=\delta_{B(H)}^{-1/2r}$ where $B(H)$ is the Borel subgroup of $H$. The definition of these representations for other groups, can be found, for example, in 
\cite{F-G} and \cite{Gao}. Depending on the relation
between the rank of $H$ and the number $r$, the representation 
$\Theta_H^{(r)}$ may or may not be generic. For example, when 
$H=GL_n$, it follows from \cite{K-P} that $\Theta_H^{(r)}$ is not generic if $r<n$, and is generic for all $r\ge n$. As an another
example, it follows from \cite{Gao}, that when $H$ is the exceptional group $G_2$, then $\Theta_H^{(r)}$ is not generic for $r=2,3,4,5,6,9$. For this group the character $\chi_\Theta^H$ is described in  Subsection \ref{prg2}.

Let $\pi$ denote the unramified representation of $G^{(r)}$ as was defined above.  Suppose that $H$ is contained in a Levi part of a
certain parabolic subgroup of $G$. Then, there is a number $k$ such
that the group $GL_1^k\times H$ is a Levi part of a parabolic subgroup
$P$ of $G$. Let $\mu$ denote an arbitrary unramified  character of $GL_1^k$, and form the induced representation $Ind_{P^{(r)}}^{
G^{(r)}}\mu\Theta_H^{(r)}\delta_{P}^{1/2}$. Suppose that $\Theta_H^{(r)}$ is not generic. Then it easily follows that 
$Ind_{P^{(r)}}^{G^{(r)}}\mu\Theta_H^{(r)}\delta_{P}^{1/2}$ is not 
generic. Our goal in this paper is to study the following 
\begin{conjecture}\label{conj1}
Let $\pi$ be as above. Ignoring several special cases mentioned below, then $\pi$
is not generic if and only if it is a subrepresentation of 
$Ind_{P^{(r)}}^{G^{(r)}}\mu\Theta_H^{(r)}\delta_{P}^{1/2}$ for
some $\mu$ and $\Theta_H^{(r)}$, and such that $\Theta_H^{(r)}$ is
not generic. 
\end{conjecture}
The special cases we need to ignore are the cases when there is a certain subgroup of a 
Levi part of some parabolic subgroup of $G$ which splits under the $r$ fold cover. This happens for the groups $Sp_{2n}^{(2)},\ SO_{2n+1}^{(2)},\ F_4^{(2)}$ and $G_2^{(3)}$. In these case we can construct induced representations not of the form mentioned in Conjecture \ref{conj1}, such that the corresponding representations $\pi$ will
not be generic.

Our main result is 
\begin{theorem}\label{th1}
Conjecture \ref{conj1} holds for the groups $GL_n^{(r)}$ when 
$r=n-1$ and  for all $r\ge n$. It also holds for the exceptional group
$G_2^{(r)}$ for all $r\ne 2$. If $r\ne 2,3,4,5,6,9$,
then $\pi$ is always generic.
\end{theorem}

Our method in proving the Theorem is essentially by direct calculation. The problem is to determine the minimal computations needed in order to derive our result. In most cases it is enough to
compute only the integral $W_{\eta}^{(r)}(e)$ where $\eta$ is the identity. Then, to obtain our result, we use intertwining operators as explained in Section \ref{prth}. When $G=GL_n$ we use the formula
given in  \cite{B-B-C-F-G} to compute $W_{\eta}^{(r)}(e)$ with $\eta=e$.  We review their formula in Section \ref{gl}. 

For the group $G_2$, a similar formula as in \cite{B-B-C-F-G}, does
not exist. Therefore we compute the integrals $W_{\eta}^{(r)}(e)$
explicitly. This is done in Section \ref{g2}. If $r\ne 2,3$, to prove the Theorem, it
is enough to compute $W_{\eta}^{(r)}(e)$ with $\eta=e$. This we do in
Proposition \ref{prop2}. In Proposition \ref{prop3} we compute 
$W_{\eta}^{(r)}(e)$ with $\eta=h(p^\alpha,p^{-\alpha})$. We do it for 
all $r$, since the computation in this case follow quite easily from the computation in Proposition \ref{prop2}. Finally, in Proposition 
\ref{prop4}, for $r=2,3$ we compute $W_{\eta}^{(r)}(e)$ for two
special choices of $\eta$. Using that, we are able to prove the
Theorem when $r=3$. Notice that $r=3$ is one of the special cases  mentioned above, so the best we can assert is that $\chi=(\chi_1,\chi_1 |\cdot |^{-1})$. See Subsection \ref{prg2}. We mention that
for the group $G_2$, excluding $r=3$, the group $H$ appearing in Conjecture \ref{conj1} is the whole group $G_2$. Indeed, the two
maximal parabolic subgroups of $G_2$ have Levi part which is the group $GL_2$. As follows from Theorem \ref{th1} for $r\ge 2$, every unramified representation of $GL_2^{(r)}$ is generic.

A similar situation occurs for the group $GL_n^{(r)}$ when $r=n-1$. In this case, the group $H$ is the group $GL_n$ itself. 

Unfortunately, these
computations are not enough to prove the Theorem for the group $G_2^{(2)}$.
The result we obtain is that $\chi$ can be one of the following three
characters. Either $\chi=(|\cdot |^{-1},|\cdot |^{-1/2})$ which, as mentioned in Subsection \ref{prg2} are the parameters of the Theta representation. Or, we get two other possibilities which are
$\chi=(1,|\cdot |^{-1/2})$ and $\chi=(|\cdot |^{-1/2},1)$. For this group, since the set
$T_0^{(2)}\backslash T^{(2)}$ contains four elements, then for $\chi$ in general position, the space of Whittaker functions of $\pi$ contains 
four linearly independent functions. It follows from our computations that for
three such functions the value $W_{\eta}^{(2)}(e)\ne 0$. Assuming that 
these three numbers are zero implies that $\chi$ is one of the three
cases we mentioned above. See Propositions \ref{prop2}, \ref{prop3}
and \ref{prop4}. 
However, when $\eta=h(p,1)$, we obtain $W_{\eta}^{(2)}(e)=0$ for all $\chi$.  Hence, to get 
information about it, we need to compute $W_{\eta}^{(2)}(g)$ for other values of $g$. 
At this point we do not know if further computations will eliminate the last two cases of $\chi$ mentioned above, or that maybe
Conjecture \ref{conj1} does not hold for $G_2^{(2)}$.

\section{Basic Definitions and Notations}\label{basic}
We first review the definition of the induced representation $Ind_{B^{(r)}}^{G^{(r)}}\chi\delta_B^{1/2}$. The reference for all this is \cite{K-P}. Let
$B$ denote the standard Borel subgroup of $G$. Let $B=TU$ where $T$ is  the maximal split torus of $G$ and $U$ is the unipotent radical of $B$.  Let 
$T^r$ denote the  subgroup of $T$ consisting of all elements whose entries are $r$ powers. Then 
$T^r$ splits inside $T^{(r)}$, the inverse image of $T$ inside $G^{(r)}$. Let $\chi$ denote an unramified character of $T$. We consider it as a character of $T^r$ by restriction. Let $K$ denote the maximal compact subgroup of $G$. It is well known that it splits under the cover, and hence can be viewed as a subgroup of $G^{(r)}$. Let $T_0^{(r)}$ denote the subgroup of $T^{(r)}$ which is generated by 
$T^r$, and by $T^{(r)}\cap K$. Then, $T_0^{(r)}$ is a maximal abelian subgroup of $T^{(r)}$. 
Thus, we can extend $\chi$ trivially from $T^r$ to $T_0^{(r)}$. Inducing up we obtain the representation $Ind_{B^{(r)}}^{G^{(r)}}\chi\delta_B^{1/2}$ which is independent of the choice of the maximal abelian subgroup of $T^{(r)}$.

Let $W_G$, or simply $W$, denote the Weyl group of $G$. We denote by $w_0$ the longest element in $W$.

Our goal is to study the set of functions given by integral \eqref{whit2}, where $\eta\in T_0^{(r)}\backslash T^{(r)}$, and $f^{(r)}$ is the unramified vector in $Ind_{B^{(r)}}^{G^{(r)}}\chi\delta_B^{1/2}$. The character $\psi_U$ is the Whittaker character of $U$. More precisely, if $\alpha_i$ for $1\le i\le l$ are the set of all positive simple roots of $G$, we can write
an element $u\in U$ as $u=\alpha_1(r_1)\ldots\alpha_l(r_l)u'$ where
$u'$ is in the commutator subgroup of $U$. Let $\psi$ denote an unramified character of $F$. Then we define 
$\psi_{U}(u)=\psi(r_1+r_2+\cdots +r_l)$. We assume that  integral
\eqref{whit2} converges in some domain of $\chi$ and admits an
analytic continuation to all $\chi$. 
We shall also denote  $\lambda_{\eta}(f^{(r)})=W_{\eta}^{(r)}(e)$. 

Given $w\in W$, We shall denote by $I_w$, the intertwining operator attached to $w$. Thus, we have
\begin{equation}\label{inter1}
I_wf^{(r)}(g)=\int\limits_{U_w} f^{(r)}(w^{-1}ug)du
\end{equation}
Here $U_w=w^{-1}Uw\cap U^-$ where $U^-=w_0Uw_0^{-1}$. This integral converges in some domain of $\chi$ and admits a meromorphic continuation.

Assume that $G=GL_n$. Then the group $G^{(r)}$ is as defined in \cite{K-P} with $c=0$. If $G=G_2$ we consider it as a subgroup of $SO_7$ and the cocycle we use is as defined for $SO_7$ in \cite{B-F-G}. In more details, if $r$ is odd we use the Hilbert symbol $(\cdot,\cdot)_r$ to form the covering group $SO_7^{(r)}$ and view it as 
subgroup of $GL_7^{(r)}$. Then, for computations in $G_2$ we restrict
the cocycle defined on $GL_7^{(r)}$ to $G_2^{(r)}$. When $r$ is even,
we do the same but with the Hilbert symbol  $(\cdot,\cdot)_{2r}$. In
both cases we get the $r$ fold cover of $G_2$. To make notations uniform, let $n=r$ if $r$ is odd, and $n=2r$ if $r$ is even. Then to form the group $G_2^{(r)}$, we use the cocycle $(\cdot,\cdot)_n$.

For both groups we shall identify elements of $G^{(r)}$ as pairs
$<g,\epsilon>$ where $g\in G$ and $\epsilon\in \mu_m$. Here $\mu_m$  
is the group of $m$ roots of unity. We shall assume that $m$ is large
enough so that $(p,p)_n=1$. Here $p$ is a generator of the maximal ideal inside the ring of integers of $F$. We shall denote $q=|p|^{-1}$.

\section{The Group $GL_n$}\label{gl}
Let $B$ be the Borel subgroup of $GL_n$ consisting of upper unipoetnt matrices. Let $T$ denote the group of all diagonal matrices in $B$.
We denote by $U$ the unipotent radical of $B$. We realize the Weyl
group $W$ of $GL_n$ as all permutation matrices in $GL_n$. 
For $1\le i\le n-1$, let $w_i$ denote the simple refection corresponding to the simple root $\alpha_i$. In matrices, we have
$w_i(i,i+1)=w_i(i+1,i)=1$ and $w_i(j,j)=1$ for all $j\ne i,i+1$. Here
$w_i(l_1,l_2)$ denotes the $(l_1,l_2)$ entry of $w_i$.

For $1\le i\le n$, let $\chi_i$ denote an  unramified character of $F^*$. Let $\chi$ denote the unramified character of $T$ defined by
$\chi(\text{diag}(a_1,\ldots, a_n))=\chi_1(a_1)\ldots \chi_n(a_n)$.
As explained in Subsection \ref{basic}, we form the induced representation $Ind_{B^{(r)}}^{G^{(r)}}\chi\delta_B^{1/2}$. 

Let $W_n^{(r)}(e)$ denote the integral \eqref{whit2} where we take
$\eta=1$ and $g=e$. We have
\begin{proposition}\label{prop1}
The integral $W_n^{(r)}(e)$ converges for all $\chi$. When $r=n-1$,
we have $W_n^{(n-1)}(e)=1-q^{-(n-1)}\chi_1^{n-1}\chi_n^{-(n-1)}(p)$.
For all $r\ge n$, we have $W_n^{(r)}(e)=1$.
\end{proposition}
\begin{proof}
To prove this Proposition we use Theorem 1 in \cite{B-B-C-F-G}. As is mentioned in \cite{M}  the result in that Theorem can be stated as follows. Consider the integral \eqref{whit2} with $\eta=1$ and $g=e$. Then perform the Iwasawa decomposition of $w_0u$. Since $f^{(r)}$
is the unramified vector, then after performing the Iwasawa 
decomposition, we are left with $f^{(r)}(t)$ which we need to evaluate. Here $t$ is a certain torus element which is obtained from the above decomposition, and whose entries are elements which are powers of $p$.  The evaluation of  $f^{(r)}(t)$ then
depends on the maximal abelian subgroup which we choose to work with. 
In \cite{B-B-C-F-G}, the maximal abelian subgroup chosen is the one
which is generated by $T^r$ and by all tori elements whose
entries are powers of $p$. Our choice is the abelian subgroup 
$T_0^{(r)}$ described in Section \ref{basic}. Nevertheless, since
the Iwasawa decomposition is an independent process which has no
relation to the maximal abelain subgroup we choose, we can then
still use the formula given in \cite{B-B-C-F-G}, but add an extra
condition which we will mention bellow. 

Using that we have from \cite{B-B-C-F-G} the following,
\begin{equation}\label{ice1}\notag
W_n^{(r)}(e)=\sum_{\Delta_r}(1-q^{-1})^{l_1(\Delta_r)}q^{-l_2(\Delta_r)/2}(-q^{-1})^{l_3(\Delta_r)}G_{\Delta_r}^{(r)}
(\chi_1\chi_2^{-1})^{k_1}(\chi_2\chi_3^{-1})^{k_2}\ldots
(\chi_{n-1}\chi_n^{-1})^{k_{n-1}}
\end{equation}
The sum is over all strict Gelfand-Tsetlin patterns $\Delta_r$ whose top row
is the row given by $\begin{pmatrix} n-1&n-2&\ldots&2&1&0\end{pmatrix}$, and such that for all $1\le i\le n-1$, the numbers $k_i$ as defined in \cite{B-B-C-F-G}
equation (13), are all zero modulo $r$. For notations related to 
Gelfand-Tsetlin patterns we refer to \cite{B-B-C-F-G}. This last 
condition on the numbers $k_i$ is the result of our choice of the
maximal ableian subgroup. Indeed, it easily follows that  
$f^{(r)}(t)$ is zero unless $t\in T^r$. Hence, we may discard all
those Gelfand-Tsetlin patterns which do not satisfy the above 
condition on the numbers $k_i$. 

The number $l_1(\Delta_r)$ is the number of numbers $a_{i,j}$ in
$\Delta_r$ such that $a_{i-1,j}\ne a_{i,j}\ne a_{i-1,j-1}$ and that
the numbers $b_{i,j}$ defined in \cite{B-B-C-F-G} equation (11), are
zero modulo $r$. In our computations we will always have $l_1(\Delta_r)=0$. The number $l_2(\Delta_r)$ is the number of numbers $a_{i,j}$ in $\Delta_r$ such that $a_{i,j}=a_{i-1,j-1}$ and that 
the numbers $b_{i,j}$ are not zero modulo $r$. For each such numbers
$a_{i,j}$ and $b_{i,j}$ we form the Gauss sum $G_{b_{i,j}}^{(r)}(1,p)$. Then, we denote by $G_{\Delta_r}^{(r)}$ the product of all those
Gauss sums where the product is over all such numbers $a_{i,j}$. 
Finally, we  denote by $l_3(\Delta_r)$ the number of all numbers 
$a_{i,j}$ which satisfies $a_{i,j}=a_{i-1,j-1}$ and that 
the numbers $b_{i,j}$ are zero modulo $r$.

With these notations, the proof of the Proposition follows easily. Indeed, we claim that when $r\ge n$, there is only one relevant 
Gelfand-Tsetlin pattern, and when $r=n-1$ there are only two. We will
prove it and for clarity, we first prove it for the group $GL_4$. To
prove that $W_4^{(r)}(e)=1$ for all $r\ge 4$, we consider all $\Delta_r$ such that
the numbers $k_i$ are all zero modulo $r$. Since the first row of
$\Delta_r$ is $\begin{pmatrix} 3&2&1&0\end{pmatrix}$ we deduce that
the only possible second row which is relevant is the row 
$\begin{pmatrix} 2&1&0\end{pmatrix}$. In this case we have $k_1=0$,
and this is the only possibility. The condition that $k_2$ will be
zero modulo $r$, implies that the third row of $\Delta_r$ is
$\begin{pmatrix} 1&0\end{pmatrix}$ and similarly, for $k_3$ to be
zero modulo $r$, the fourth row of $\Delta_r$ must be $0$. Thus, we
obtain one Gelfand-Tsetlin pattern which satisfies the required
conditions. Since $a_{i,j}=a_{i-1,j}$ for all $(i,j)$ we deduce that
$l_m(\Delta_r)=0$ for all $m=1,2,3$. We refer to this pattern as the trivial pattern.  We thus obtain that  $W_4^{(r)}(e)=1$. For
general values of $r\ge n$, we obtain the same result, namely that only the trivial pattern contributes. Hence, $W_n^{(r)}(e)=1$ for all
$n$.

Next consider the case when $r=n-1$. Clearly we will have the
contribution of the trivial
pattern which contributes one to the value of $W_n^{(n-1)}(e)$. 
There is one more relevant pattern. 
Indeed, when $n=4$ the only other option
for the second row is $\begin{pmatrix} 3&2&1\end{pmatrix}$ which 
implies that $k_1=3$. With that we deduce that only $\begin{pmatrix}
3&1\end{pmatrix}$ is relevant and it gives $k_2=3$. Then the last
row in $\Delta_3$ is the number three which implies $k_3=3$. For this
pattern we have $a_{1,1}=a_{2,2}=a_{3,3}=3$, $a_{1,2}=2$ and 
$a_{1,3}=1$. These are the only numbers $a_{i,j}$ which can contribute
to the numbers $l_m(\Delta_r)$. We have $b_{1,1}=3;\ b_{1,2}=b_{2,2}
=1$ and $b_{1,3}=b_{3,3}=2$. Hence, we have $l_1(\Delta_3)=0$. Only
$a_{1,1}$ will contribute to $l_3(\Delta_3)$ and all other four
numbers contribute to $l_2(\Delta_3)$. Hence $l_2(\Delta_3)=4$ and
$l_3(\Delta_3)=1$. Finally, since for the four numbers contributing
to $l_2(\Delta_3)$ two of the values of the corresponding $b_{i,j}$ are ones and two are twos, we deduce that  $G_{\Delta_r}^{(3)}=
(G_1^{(3)}(1,p))^2(G_2^{(3)}(1,p))^2=1$. Hence, the contribution to 
$W_4^{(3)}(e)$ from this pattern is given by 
$$(1-q^{-1})^0(q^{-4/2})(-q^{-1})^1(\chi_1\chi_2^{-1})^3
(\chi_2\chi_3^{-1})^3(\chi_3\chi_4^{-1})^3=-q^{-3}\chi_1^3
\chi_4^{-3}$$
Combining with the contribution of one from the trivial pattern, the
Proposition follows for the case $n=4$. For any $n$,  beside the trivial pattern we also have one more relevant pattern. It is the one
such that $a_{i,i}=n-1$ for all $1\le i\le n-1$, and 
$a_{i,i+j}=n-j-i$ for all $1\le i\le n-j-1$. For this pattern we
have $k_l=n-1$ for all $l$. As in the case when $n=4$, we have no
$a_{i,j}$ which will contribute to $l_1(\Delta_{n-1})$ which is
then equal zero. The only relevant numbers to compute the rest of
the factors are, $a_{1,1}=n-1;\ b_{1,1}=n-1$,\ $a_{i,i}= n-1;\ b_{i,i}
=i-1$ for $2\le i\le n-1$, and $a_{1,i}=n-i;\ b_{1,i}=n-i$ for 
$2\le i\le n-1$. In the first case, from $a_{1,1}$ we get a contribution to $l_3(\Delta_{n-1})$. From the second case, we get the contribution $q^{-1/2}G_{i-1}^{(n-1)}(1,p)$ for all $2\le i\le n-1$.
From the third case we get $q^{-1/2}G_{n-i}^{(n-1)}(1,p)$ for all $2\le i\le n-1$. Hence, $l_3(\Delta_{n-1})=1$ and $l_2(\Delta_{n-1})=
2(n-2)$. 
Thus, this pattern contributes  
$$q^{-2(n-2)/2}(\prod_{i=2}^{n-1}G_{i-1}^{(n-1)}(1,p))^2(-q^{-1})(\chi_1\chi_2^{-1})^{n-1}
\dots (\chi_{n-1}\chi_n^{-1})^{n-1}=-q^{-(n-1)}(\chi_1\chi_n^{-1})^{n-1}$$
From this the Proposition follows.

\end{proof}

\section{The group $G_2$}\label{g2}
For computations it will be convenient to use the matrix realization 
of $G_2$ as given in \cite{H-R-T}. Thus, $T$ is generated by all
$h(k_1,k_2)=\text{diag}(k_1,k_2,k_1k_2^{-1},1,k_1^{-1}k_2,k_2^{-1},k_1^{-1})$. Recall that $n=r$ if $r$ is odd, and $n=2r$ if $r$ is even.
With these notations we have the following identity
$<h(k_1,k_2),1><h(r_1,r_2),1>=(k_1,r_1^{-2}r_2)_n(k_2,r_1r_2^{-2})_n<h(k_1r_1,k_2r_2),1>$. This implies that 
\begin{equation}\label{cocycle1}
<h(\epsilon_1^{-1},\epsilon_2^{-1}),1><h(p^\alpha,p^\beta),1>
<h(\epsilon_1,\epsilon_2),1>=(\epsilon_1,p^{4\alpha-2\beta})_n
(\epsilon_2,p^{-2\alpha+4\beta})_n<h(p^\alpha,p^\beta),1>
\end{equation}

Henceforth, when there is no confusion, we shall write $g$ for 
$<g,1>$. Let $\chi$ denote an unramified character of $T$. With the 
above parametrization, we write $\chi(h(k_1,k_2))=\chi_1(k_1)\chi_2(k_2)$. Here $\chi_i$ are two unramified characters of $F^*$. 
As explained in Section \ref{basic}, using $\chi$ we can construct
the induced representation $Ind_{B^{(r)}}^{G^{(r)}}\chi\delta_B^{1/2}$.

Let $a$ and $b$ denote the two simple roots of $G_2$, where $a$ is the
short root. For a root $\alpha$, we shall denote by $x_\alpha(l)$ the one dimensional unipotent subgroup associated with this root. Matrix realization of $x_a(l)$ and $x_b(l)$ are given in \cite{H-R-T}. The
other positive roots of $G_2$ are $a+b,\ 2a+b,\ 3a+b$ and $3a+2b$. We
shall denote by $w_a$ and $w_b$ the two simple reflection in $W$.
In the proof of the following Proposition we will repeatedly use
commutation relations between elements of the group $G_2$. Using
the matrix realization of $x_a(l),x_b(l)$ and the simple reflections
$w_a$ and $w_b$, given in \cite{H-R-T},  it is easy to deduce these commutation relations. 
Hence,we will not specify them, and leave it to reader to verify 
them.

We first compute integral \eqref{whit2} with $\eta=1$ and $g=1$. We shall denote it by $W^{(r)}(e)$. Also, when there is no confusion, we shall write $\chi_i$ for $\chi_i(p)$. We have
\begin{proposition}\label{prop2}
Integral $W^{(r)}(e)$ converges in the domain $|\chi_1\chi_2|<q^{1/2}$. We have,
\begin{equation}\label{two}
W^{(2)}(e)=1+(1-q^{-1})q^{-1}\chi_1^2(1+\chi_2^2)-q^{-2}\chi_1^2(\chi_1^2
+\chi_2^4)-(1-q^{-1})q^{-2}\chi_1^4\chi_2^2(1+
\chi_2^2)+q^{-4}\chi_1^6\chi_2^4
\end{equation}
\begin{equation}\label{three}
W^{(3)}(e)=(1-q^{-1}\chi_1\chi_2^{-1})(1-q^{-1}\chi_1^2\chi_2)
(1-q^{-1}\chi_1\chi_2^2)(1+q^{-1}\chi_1^2\chi_2+q^{-1}\chi_1\chi_2^2)
\end{equation}
\begin{equation}\label{four}
W^{(4)}(e)=1-q^{-2}\chi_1^4
\end{equation}
\begin{equation}\label{five}
W^{(5)}(e)=1-q^{-3}\chi_1^5\chi_2^5
\end{equation}
\begin{equation}\label{six}
W^{(6)}(e)=1-q^{-2}\chi_1^2\chi_2^4
\end{equation}
\begin{equation}\label{nine}
W^{(9)}(e)=1-q^{-3}\chi_1^6\chi_2^3
\end{equation}
For all other values of $r$ we have $W^{(r)}(e)=1$.
\end{proposition}
\begin{proof}
Conjugating the Weyl element to the right, $W^{(r)}(e)$ is equal to
\begin{equation}\label{g21}\notag
I=\int\limits_{F^5}f_{W_a}(x_{-b}(r_1)x_{-(a+b)}(r_2)
x_{-(2a+b)}(r_3)x_{-(3a+b)}(m_1)x_{-(3a+2b)}(m_2))\psi(m_1)dr_idm_j
\end{equation}
where 
\begin{equation}\label{g22}
f_{W_a}(g)=\int\limits_F f^{(r)}(w_ax_a(l)g)\psi(l)dl\notag
\end{equation}

Write $I=I_1+I_2$, where $I_1$ is the contribution to $I$ from the
integration domain $|m_1|\le 1$, and $I_2$ the contribution from $|m_1|>1$. Using
the right invariance property of $f_{W_a}$, we obtain
\begin{equation}\label{g23}
I_1=\int\limits_{F^4}f_{W_a}(x_{-b}(r_1)x_{-(a+b)}(r_2)
x_{-(2a+b)}(r_3)x_{-(3a+2b)}(m_2))dr_idm_2
\end{equation}
For all $|k|\le 1$, we have $f_{W_a}(g)=f_{W_a}(gx_{3a+b}(k))$. Using
that in integral \eqref{g23} and conjugating $x_{3a+b}(k)$ to the
left, we obtain the integral $\int\psi(r_3k)dk$ as an inner integration. In the above we used the left invariant property of 
$f_{W_a}$ and relevant commutation relations in $G_2$. The integration domain is over $|k|\le 1$. Hence, if $|r_3|>1$ we
get zero contribution to $I_1$. Thus, 
\begin{equation}\label{g24}\notag
I_1=\int\limits_{F^3}f_{W_a}(x_{-b}(r_1)x_{-(a+b)}(r_2)
x_{-(3a+2b)}(m_2))dr_idm_2
\end{equation}
Write $I_1=I_{11}+I_{12}$ where $I_{11}$ is the contribution from
the integration domain $|m_2|\le 1$ and $I_{12}$ is the contribution from
$|m_2|>1$. To compute $I_{11}$ we repeat the same process as we
did above with the variable $r_3$. Thus, using $x_{2a+b}(k)$ with
$|k|\le 1$ we deduce that the contribution to $I_{11}$ from the domain
$|r_2|>1$ is zero. Similarly, using $x_{a+b}(k)$ with
$|k|\le 1$ we deduce that the contribution to $I_{11}$ from the domain
$|r_1|>1$ is zero. Thus, we have $I_{11}=f_{W_a}(e)$. 

Next, to compute $I_{12}$ we perform the Iwasawa decomposition of
$x_{-(3a+2b)}(m_2)$ with $|m_2|>1$. We have $x_{-(3a+2b)}(m_2)=
x_{3a+2b}(m_2^{-1})h(m_2^{-1},m_2^{-1})k$ where $k$ is in the maximal
compact subgroup. Plugging this in integral $I_{12}$ we obtain 
\begin{equation}\label{g25}\notag
I_{12}=\int\limits_{|m_2|>1}
\int\limits_{F^2}f_{W_a}(x_{-b}(r_1)x_{-(a+b)}(r_2)
x_{3a+2b}(m_2^{-1})h(m_2^{-1},m_2^{-1}))dr_idm_2
\end{equation}
Conjugating the two right most elements to the left, using the left
invariant properties of $f_{W_a}$, we obtain the integral 
\begin{equation}\label{g26}\notag
I_{12}=\int\limits_{|m_2|>1}
\int\limits_{F^2}|m_2|^2f_{W_a}(h(m_2^{-1},m_2^{-1})
x_{-b}(r_1)x_{-(a+b)}(r_2))\psi(m_2r_2^2)dr_idm_2
\end{equation}
In the above, the factors $|m_2|^2$ and $\psi(m_2r_2^2)$ are derived from the change of variables which is performed during the conjugation. Notice that the matrix $h(m_2^{-1},m_2^{-1})$ commutes
with the matrix $x_a(l)$. Hence, we can perform the same steps we
did in the computation of $I_{11}$, to deduce that the contribution
to $I_{12}$ is zero unless $|r_i|\le 1$ for $i=1,2$. Thus, we have
\begin{equation}\label{g27}\notag
I_{12}=\int\limits_{|m_2|>1}
|m_2|^2f_{W_a}(h(m_2^{-1},m_2^{-1}))\int\limits_{|r_2|\le 1}\psi(m_2r_2^2)dr_2dm_2
\end{equation}
We shall not compute this integral, since it will cancel with another
contribution. However, direct calculation implies that this integral
converges in the domain $|\chi_1\chi_2|<q^{1/2}$.

Next we compute $I_2$. Writing $m_1=p^{-m}\epsilon$ with $m>1$ and 
$|\epsilon|=1$, we obtain that $I_2$ is equal to
\begin{equation}\label{g28}\notag
\sum_{m=1}^\infty q^m\int\limits_{F^4}\int\limits_{|\epsilon|=1}
f_{W_a}(x_{-b}(r_1)x_{-(a+b)}(r_2)
x_{-(2a+b)}(r_3)x_{-(3a+b)}(p^{-m}\epsilon)x_{-(3a+2b)}(m_2))\psi(p^{-m}\epsilon)dr_idm_2d\epsilon
\end{equation}
Notice that $f_{W_a}(g)=f_{W_a}(gh(\epsilon,\epsilon))=
f_{W_a}(h(\epsilon,\epsilon)g)$ for $|\epsilon|=1$. The last equality
is obtained from the property $f^{(r)}(h(\epsilon_1,\epsilon_2)g)=
f^{(r)}(g)$ for all $|\epsilon_i|=1$. We also have the equality  $x_{-(3a+b)}(p^{-m}\epsilon)=
h(\epsilon^{-1},\epsilon^{-1})x_{3a+b}(p^m)h(\epsilon,\epsilon)$. Using that in the above integral, we obtain the integral 
$\int \psi(p^{-m}\epsilon)d\epsilon$ as an inner integration. Here,
we integrate over $|\epsilon|=1$. Hence, we obtain zero integration
if $m>1$. When $m=1$ we get $-q^{-1}$. Hence, 
\begin{equation}\label{g29}\notag
I_2=-\int\limits_{F^4}
f_{W_a}(x_{-b}(r_1)x_{-(a+b)}(r_2)
x_{-(2a+b)}(r_3)x_{-(3a+b)}(p^{-1})x_{-(3a+2b)}(m_2))dr_idm_2
\end{equation}
Using the Iwasawa decomposition $x_{-(3a+b)}(p^{-1})=x_{3a+b}(p)
h(p,1)k$ where $k$ is in the maximal compact subgroup, we obtain after 
conjugation 
\begin{equation}\label{g210}\notag
I_2=-q\int\limits_{F^4}
f_{W_a}(h(p,1)x_{-b}(r_1)x_{-(a+b)}(r_2)
x_{-(2a+b)}(r_3)x_{-(3a+2b)}(m_2))\psi(-r_3)dr_idm_2
\end{equation}
We claim that the contribution to $I_2$ from the integration domain 
$|r_3|\le 1$, is zero. Denote this contribution by $I_2^0$. Then, in this domain, the character 
$\psi(-r_3)$ is trivial. Using the left and right invariant
property of $f_{W_a}$ under $h(\epsilon,\epsilon)$, we obtain a
cocycle contribution which results from the matrix multiplication  $h(\epsilon^{-1},\epsilon^{-1})h(p,1)h(\epsilon,\epsilon)$. Using \eqref{cocycle1}, this conjugation produces the
Hilbert symbol $(\epsilon,p)^2_n$. Thus, we obtain $I_2^0=
(\epsilon,p)^2_nI_2^0$. For all $n$, we can find an $\epsilon$ such that $(\epsilon,p)^2_n$ is not trivial. Hence $I_2^0=0$. Hence, in integral $I_2$, we may restrict to the integration domain where $|r_3|>1$. This is equal to
\begin{equation}\label{g211}\notag
-q\sum_{m=1}^\infty q^m\int\limits_{F^3}
\int\limits_{|\epsilon|=1}
f_{W_a}(h(p,1)x_{-b}(r_1)x_{-(a+b)}(r_2)
x_{-(2a+b)}(p^{-m}\epsilon)x_{-(3a+2b)}(m_2))\psi(-p^{-m}\epsilon)dr_idm_2d\epsilon
\end{equation}
We have $x_{-(2a+b)}(p^{-m}\epsilon)=h(\epsilon^{-1},\epsilon^{-1})
x_{-(2a+b)}(p^{-m})h(\epsilon,\epsilon)$. Conjugating the tori elements to the left and right, we obtain the integral 
$\int (\epsilon,p)^{-2}_n\psi(-p^{-m}\epsilon)d\epsilon$ as
inner integration. This integral is zero if $m>1$, and for $m=1$
it is equal to $q^{-1/2}\overline{G_2^{(n)}(1,p)}$. Here $G_2^{(n)}(1,p)$ is the normalized Gauss sum, so that $|G_2^{(n)}(1,p)|=1$.
Performing the Iwasawa decomposition $x_{-(2a+b)}(p^{-1})=
x_{2a+b}(p)h(p^2,p)k$ where $k$ is in the compact, we obtain 
\begin{equation}\label{g212}\notag
I_2=-q^{11/2}\overline{G_2^{(n)}(1,p)}\int\limits_{F^3}
f_{W_a}(h(p^3,p)x_{-b}(r_1)x_{-(a+b)}(r_2)
x_{-(3a+2b)}(m_2))\psi(-m_2)\psi(2pr_2)dr_idm_2
\end{equation}
Write $I_2$ as a sum of four integrals $I_{2i}$  according to the
contribution from the various integration domains  of the variables $r_2$ and $m_2$. 

We start with $I_{21}$ which  correspond to the integration domain 
$|r_2|,|m_2|\le 1$. Breaking the domain of integration of $r_1$ in a 
similar way, we obtain
\begin{equation}\label{g213}\notag
I_{21}=-q^{11/2}\overline{G_2^{(n)}(1,p)}f_{W_a}(h(p^3,p))
-q^{11/2}\overline{G_2^{(n)}(1,p)}
\int\limits_{|r_1|>1}f_{W_a}(h(p^3,p)x_{-b}(r_1))dr_1
\end{equation}
The second summand is equal to
\begin{equation}\label{g214}\notag
-q^{11/2}\overline{G_2^{(n)}(1,p)}\sum_{m=1}^\infty q^m
\int\limits_{|\epsilon|=1}f_{W_a}(h(p^3,p)h(1,p^m\epsilon^{-1}))
d\epsilon
\end{equation}
In the above we also performed the Iwasawa decomposition for 
$x_{-b}(r_1)$ with $|r_1|>1$. It follows from the cocycle computation right before \eqref{cocycle1}, that the factorization 
$h(1,p^m\epsilon^{-1})=h(1,p^m)h(1,\epsilon^{-1})$ produces the Hilbert symbol $(\epsilon,p)^{2m}_n$.
Hence, the above is equal to
\begin{equation}\label{g215}\notag
-q^{11/2}\overline{G_2^{(n)}(1,p)}\sum_{m=1}^\infty q^m
f_{W_a}(h(p^3,p^{m+1}))\int\limits_{|\epsilon|=1}
(\epsilon,p)^{2m}_nd\epsilon
\end{equation}
Clearly, $f_{W_a}(h(p^3,p^{m+1}))=0$ for all $m\ge 3$. When $m=1$ the
integral over $\epsilon$ is zero for all $n$. So we are left with $m=2$. Thus
\begin{equation}\label{g216}
I_{21}=-q^{11/2}\overline{G_2^{(n)}(1,p)}f_{W_a}(h(p^3,p))
-q^{15/2}\overline{G_2^{(n)}(1,p)}f_{W_a}(h(p^3,p^3))
\int\limits_{|\epsilon|=1}
(\epsilon,p)^4_nd\epsilon
\end{equation}
Next we denote by $I_{22}$ the contribution to $I_2$ from the integration
domain $|r_2|>1$ and $|m_2|\le 1$. Plug in the Iwasawa decomposition 
$x_{-(a+b)}(r_2)=x_{a+b}(r_2^{-1})h(r_2^{-1},r_2^{-2})k$, and 
conjugate the matrices to the left, we obtain
\begin{equation}\label{g217}\notag
I_{22}=-q^{11/2}\overline{G_2^{(n)}(1,p)}\int\limits_{|r_2|>1}
\int\limits_{F}
f_{W_a}(h(p^3,p)h(r_2^{-1},r_2^{-2})x_{-b}(r_1))\psi(2pr_2)
\psi(p^2r_2^2r_1)|r_2|^3dr_i
\end{equation}
Assume $|r_1|\le 1$. Then we obtain as inner integration
$\int \psi(p^2r_2^2r_1)dr_1$ integrated over $|r_1|\le 1$. 
It follows
from the support of the function $f_{W_a}$ that $|p^2r_2|\le 1$. 
Hence $|r_2|=q,q^2$. If $|r_2|=q^2$, we get as an inner integration 
$\int \psi(p^{-2}r_1)dr_1$ which is clearly zero. If $|r_2|=q$, we write
$r_2=p^{-1}\epsilon$ with $|\epsilon|=1$. Using \eqref{cocycle1} we obtain 
$$-q^{19/2}\overline{G_2^{(n)}(1,p)}
f_{W_a}(h(p^4,p^3))\int_{|\epsilon|=1}(\epsilon,p)_n^6 d\epsilon$$
Hence, for this term to be nonzero we must have $r=3$. However,
conjugating in $f_{W_a}$ by $h(\eta,\eta)$ we obtain that 
$f_{W_a}(h(p^4,p^3))$ is not zero only if $(\eta,p)_n^{14}=1$. This is
not the case if $r=3$, and hence we get zero contribution.

Hence,
in integral $I_{22}$ we may restrict to the domain $|r_1|>1$. Performing the Iwasawa decomposition in $r_1$ we obtain
\begin{equation}\label{g218}\notag
I_{22}=-q^{11/2}\overline{G_2^{(n)}(1,p)}\int\limits_{|r_1|,|r_2|>1}
f_{W_a}(h(p^3,p)h(r_2^{-1},r_2^{-2})
h(1,r_1^{-1}))\psi(2pr_2)
\psi(p^2r_2^2r_1)|r_2|^3dr_i
\end{equation}
Write $r_2=p^{-m}\epsilon$ and $r_1=p^{-k}\eta$ where $|\epsilon|=
|\eta|=1$. Then, using \eqref{cocycle1}, and  taking into an account the Hilbert symbol produced by the various factorizations, we obtain
\begin{equation}\label{g219}\notag
I_{22}=-q^{11/2}\overline{G_2^{(n)}(1,p)}\sum_{m,k=1}^\infty
q^{4m+k}f_{W_a}(h(p^{m+3},p^{2m+k+1}))\times
\end{equation}
$$\times \int\limits_{|\epsilon|=
|\eta|=1}(\eta,p)^{2k}_n(\epsilon,p)_n^{6m+6k}\psi(2p^{1-m}\epsilon)
\psi(p^{2-2m-k}\epsilon^2\eta)d\epsilon d\eta$$
Changing variables $\eta\to \eta\epsilon^{-2}$ we deduce from
the integral over $\eta$ that the only nontrivial contribution is from 
$m=k=1$. When this happens, we obtain the integral
$\int (\eta,p)^{2}_n\psi(p^{-1}\eta)d\eta$ as inner integration. 
This integral is equal to $q^{-1/2}G_2^{(n)}(1,p)$. 
Thus, the Gauss sum is cancelled, and we obtain 
\begin{equation}\label{g220}
I_{22}=-q^{10}f_{W_a}(h(p^{4},p^{4}))\int\limits_{|\epsilon|=1}
(\epsilon,p)_n^8d\epsilon
\end{equation}

Next, in $I_2$ we consider the integration domain $|m_2|>1$. Performing the Iwasawa decomposition $x_{-(3a+2b)}(m_2)=
x_{3a+2b}(m_2^{-1})h(m_2^{-1},m_2^{-1})k$, and conjugating those 
matrices to the left, we obtain the contribution of
\begin{equation}\label{g221}
-q^{11/2}\overline{G_2^{(n)}(1,p)}\times
\end{equation}
$$\times\int\limits_{|m_2|>1}\int\limits_{F^2}
f_{W_a}(h(p^3,p)h(m_2^{-1},m_2^{-1})x_{-b}(r_1)x_{-(a+b)}(r_2))
\psi(-(1-pr_2)^2m_2)|m_2^2|dr_idm_2$$
Let $I_{23}$ denote the contribution to integral \eqref{g221}
from the integration domain $|r_2|\le1$. Then, $|pr_2|<1$, and hence 
$1-pr_2$ is a unit which is an $n$ power. Hence, changing variables 
in $m_2$, we obtain 
\begin{equation}\label{g222}\notag
I_{23}=-q^{11/2}\overline{G_2^{(n)}(1,p)}
\int\limits_{|m_2|>1}\int\limits_{F}
f_{W_a}(h(p^3,p)h(m_2^{-1},m_2^{-1})x_{-b}(r_1))
\psi(m_2)|m_2^2|dr_1dm_2
\end{equation}
Write $m_2=p^{-m}\epsilon$. Using \eqref{cocycle1} we obtain
\begin{equation}\label{g223}\notag
I_{23}=-q^{11/2}\overline{G_2^{(n)}(1,p)}\sum_{m=1}^\infty q^{3m}
\int\limits_{F}
f_{W_a}(h(p^{m+3},p^{m+1})x_{-b}(r_1))
\int\limits_{|\epsilon|=1}(\epsilon,p)^{2m}_n\psi(p^{-m}\epsilon)
d\epsilon dr_1
\end{equation}
The inner integration is nonzero only if $m=1$, and for $m=1$ it is
equal to $q^{-1/2}G_2^{(n)}(1,p)$. Thus, the Gauss sum is cancelled,
and we obtain 
\begin{equation}\label{g2231}\notag
I_{23}=-q^8
\int\limits_{F}f_{W_a}(h(p^4,p^2)x_{-b}(r_1))dr_1
\end{equation}
Breaking the integration domain of $r_1$ to $|r_1|\le 1$ and $|r_1|>1$, we obtain 
\begin{equation}\label{g224}
I_{23}=-q^8f_{W_a}(h(p^4,p^2))-q^{10}f_{W_a}(h(p^4,p^4))
\int\limits_{|\epsilon|=1}(\epsilon,p)^4_nd\epsilon
\end{equation}

We are left with integral $I_{24}$ which corresponds in integral $I_2$ to the integration domain over $|r_2|,|m_2|>1$. Performing the Iwasawa decomposition $x_{-(a+b)}(r_2)=x_{a+b}(r_2^{-1})h(r_2^{-1},r_2^{-2})
k$ in integral \eqref{g221}, we obtain after conjugation
\begin{equation}\label{g225}
I_{24}=-q^{11/2}\overline{G_2^{(n)}(1,p)}
\int\limits_{|r_2|,|m_2|>1}\int\limits_{F}
f_{W_a}(h(p^3,p)h(m_2^{-1},m_2^{-1})
h(r_2^{-1},r_2^{-2})x_{-b}(r_1))\times
\end{equation}
$$\times\psi(-(1-pr_2)^2m_2)\psi(p^2r_2^2r_1)|r_2^3m_2^2|dr_idm_2$$
Let $I_{241}$ denote the contribution to $I_{24}$ from the integration domain $|r_1|\le 1$. Writing $r_2=p^{-m}\eta$ and $m_2=p^{-k}\epsilon$, we obtain using \eqref{cocycle1},
\begin{equation}\label{g226}
I_{241}=-q^{11/2}\overline{G_2^{(n)}(1,p)}
\sum_{m,k =1}^\infty q^{4m+3k}
f_{W_a}(h(p^{k+m+3},p^{k+2m+1}))\times
\end{equation}
$$\times \int\limits_{|\epsilon|,|\eta|=1}
(\epsilon,p)_n^{2k+6m}(\eta,p)^{6m}
\psi((1-p^{1-m}\eta)^2p^{-k}\epsilon)d\epsilon d\eta
\int\limits_{|r_1|\le 1}\psi(p^{2-2m}r_1)dr_1$$
From the integration over $r_1$ we deduce that the only nontrivial
contribution to the above integral is when $m=1$. For $|\nu|=1$, using \eqref{cocycle1} the multiplication  
$h(\nu,\nu)^{-1}h(p^{k+4},p^{k+3})h(\nu,\nu)$ produces the Hilbert
symbol $(\nu,p)_n^{14+4k}$. Hence $f_{W_a}(h(p^{k+4},p^{k+3}))=0$ unless $(\nu,p)_n^{14+4k}=1$. In \eqref{g226} with $m=1$ we consider
first the contribution when $\eta\in 1+{\mathcal P}$. Here ${\mathcal P}$ is the maximal ideal in the ring of integers of $F$. Since
$(\eta,p)_n=1$ for such $\eta$, we get the contribution of
\begin{equation}\label{g227}\notag
-q^{19/2}\overline{G_2^{(n)}(1,p)}
\sum_{k =1}^\infty q^{3k}
f_{W_a}(h(p^{k+4},p^{k+3}))\int\limits_{|\epsilon|=1}
(\epsilon,p)_n^{2k+6}\int\limits_{|t|<1}\psi(t^2p^{-k}\epsilon)dtd\epsilon
\end{equation}
The inner integration is equal to
\begin{equation}\label{g228}\notag
\int\limits_{|\epsilon|=1}
(\epsilon,p)_n^{2k+6}\sum_{m=1}^\infty q^{-m}
\int\limits_{|\nu|=1}\psi(p^{2m-k}\nu^2\epsilon)d\nu d\epsilon
\end{equation}
Changing variables $\epsilon\to \epsilon\nu^{-2}$, we obtain 
$\int\limits(\nu,p)_n^{-4k-12}d\nu$ as inner integration. Since we
have $(\nu,p)_n^{14+4k}=1$, this integral is equal to 
$\int\limits(\nu,p)_n^2d\nu$, which is zero. Thus we are left with the contribution to $I_{241}$ from the integration domain $\eta\notin 
1+{\mathcal P}$. Thus, integral $I_{241}$ is equal to
\begin{equation}\label{g229}
-q^{19/2}\overline{G_2^{(n)}(1,p)}
\sum_{k=1}^\infty q^{3k}
f_{W_a}(h(p^{k+4},p^{k+3}))\int\limits_{|\epsilon|=1\ ,\eta\notin 
1+{\mathcal P}}(\epsilon,p)_n^{2k+6}(\eta,p)_n^6\psi((1-\eta)^2p^{-k}
\epsilon)d\eta d\epsilon
\end{equation}
For such $\eta$, we have $|1-\eta|=1$. Changing variables in $\epsilon$, we obtain $\int (\epsilon,p)_n^{2k+6}\psi(p^{-k}\epsilon)d\epsilon$
as inner integration. This is not zero only if $k=1$. Hence
\begin{equation}\label{g230}\notag
I_{241}=-q^{25/2}\overline{G_2^{(n)}(1,p)}
f_{W_a}(h(p^5,p^4))\int\limits_{\eta\notin 
1+{\mathcal P}}(\eta,p)_n^6(1-\eta,p)_n^{-16}d\eta
\int\limits_{|\epsilon|=1}
(\epsilon,p)_n^8\psi(p^{-1}\epsilon)d\epsilon
\end{equation}
Recall from above that $f_{W_a}(h(p^5,p^4))=0$ unless $(\nu,p)_n^{18}
=1$. Hence, 
\begin{equation}\label{g231}
I_{241}=-q^{25/2}\overline{G_2^{(n)}(1,p)}
f_{W_a}(h(p^5,p^4))\int\limits_{\eta\notin 
1+{\mathcal P}}(\eta,p)_n^6(1-\eta,p)_n^2d\eta
\int\limits_{|\epsilon|=1}
(\epsilon,p)_n^8\psi(p^{-1}\epsilon)d\epsilon
\end{equation}
The condition $(\nu,p)_n^{18}=1$ implies that $n=r=3$ or $n=r=9$. In
the first case the integration over $\epsilon$ is equal to
$q^{-1/2}G_2^{(3)}(1,p)$. The Gauss sum is cancelled and we obtain 
\begin{equation}\label{g232}\notag
I_{241}=-q^{12}
f_{W_a}(h(p^5,p^4))\int\limits_{\eta\notin 
1+{\mathcal P}}(1-\eta,p)_3^2d\eta
\end{equation}
Suppose that $\{\epsilon_i\}$ is a set of representatives for ${\mathcal 
O}^*/1+{\mathcal P}$, and assume that $\epsilon_1=1$. Then
$$\int\limits_{\eta\notin 
1+{\mathcal P}}(1-\eta,p)_3^2d\eta=\sum_{\epsilon_i\ne 1}
\ \int\limits_{\eta\in 
\epsilon_i+{\mathcal P}}(1-\eta,p)_3^2d\eta=
\sum_{\epsilon_i\ne 1}
\ \int\limits_{\eta\in 
\epsilon_i+{\mathcal P}}(\eta,p)_3^2d\eta=
\int\limits_{\eta\notin 
1+{\mathcal P}}(\eta,p)_3^2d\eta$$
Here we used the fact that if $\{\epsilon_i\}$ with $i\ne 1$ is a set of nontrivial
representatives then so is $\{\epsilon_i +1\}$ with $i\ne 1$. Since
$$\int\limits_{\eta\notin 
1+{\mathcal P}}(\eta,p)_3^2d\eta=\int\limits_{|\eta|=1}(\eta,p)_3^2d\eta-\int\limits_{\eta\in  1+{\mathcal P}}(\eta,p)_3^2d\eta=
\int\limits_{\eta\in  1+{\mathcal P}}d\eta=-q^{-1}$$
we obtain that when $r=3$,
\begin{equation}\label{g233}
I_{241}=q^{11} f_{W_a}(h(p^5,p^4))
\end{equation}
When $n=r=9$, the integral over $\epsilon$ in integral $I_{241}$
is equal to $q^{-1/2}G_8^{(9)}(1,p)$. At the end of the proof of this Proposition, we will prove that 
\begin{equation}\label{g234}
\int\limits_{\eta\notin 
1+{\mathcal P}}(\eta,p)_9^6(1-\eta,p)_9^2d\eta=
q^{-1/2}G_1^{(9)}(1,p)G_6^{(9)}(1,p)G_2^{(9)}(1,p)
\end{equation}
Cancelling the relevant Gauss sums, we obtain the contribution 
\begin{equation}\label{g235}
I_{241}=-q^{23/2} G_6^{(9)}(1,p)f_{W_a}(h(p^5,p^4))
\end{equation}

To complete the computation of $I_{24}$ we need to consider the 
contribution to \eqref{g225} from the integration domain $|r_1|
>1$. We denote this contribution by $I_{242}$.
Performing the Iwasawa decomposition of $x_{-b}(r_1)$, we obtain
\begin{equation}\label{g236}\notag
I_{242}=-q^{11/2}\overline{G_2^{(n)}(1,p)}
\int\limits_{|r_1|,|r_2|,|m_2|>1}
f_{W_a}(h(p^3,p)h(m_2^{-1},m_2^{-1})
h(r_2^{-1},r_2^{-2})h(1,r_1^{-1}))\times
\end{equation}
$$\times\psi(-(1-pr_2)^2m_2)\psi(p^2r_2^2r_1)|r_2^3m_2^2|dr_idm_2$$
Write $r_1=p^{-l}\nu,\ r_2=p^{-k}\eta$ and $m_2=p^{-m}\epsilon$.
Using \eqref{cocycle1} this is equal to
\begin{equation}\label{g237}\notag
-q^{11/2}\overline{G_2^{(n)}(1,p)}\sum_{m,k,l=1}^\infty q^{3m+4k+l}
f_{W_a}(h(p^{m+k+3},p^{m+2k+l+1}))\times
\end{equation}
$$\times\int\limits(\nu,p)_n^{2l}(\eta,p)_n^{6k+6l}(\epsilon,p)_n
^{2m+6k+2l}\psi(p^{2-2k-l}\eta^2\nu)\psi((1-p^{1-k}\eta)^2p^{-m}
\epsilon)d\nu d\eta d\epsilon$$
Here $\epsilon,\ \eta$ and $\nu$ are integrated over the groups of units in the ring of integers of $F$. Changing variables $\nu\to
\nu\eta^2$, we obtain from the integral over $\nu$ that the only non
trivial contribution is when $k=l=1$. In this case we also obtain a
cancellation of the Gauss sum, and hence
\begin{equation}\label{g238}\notag
I_{242}=-q^{10}\sum_{m=1}^\infty q^{3m}
f_{W_a}(h(p^{m+4},p^{m+4}))\int\limits_{|\epsilon|,|\eta|=1}
(\eta,p)_n^8(\epsilon,p)_n
^{2m+8}\psi((1-\eta)^2p^{-m}
\epsilon) d\eta d\epsilon
\end{equation}
In the integration domain $\eta\notin 1+{\mathcal P}$ we argue as in
\eqref{g229}, to deduce first that the only nontrivial contribution is when $m=1$. When this is the case  we obtain
\begin{equation}\label{g239}
I_{2421}=-q^{13}
f_{W_a}(h(p^5,p^5))
\int\limits_{|\epsilon|=1}(\epsilon,p)_n^{10}
\psi(p^{-1}\epsilon)d\epsilon 
\left (\int\limits_{|\eta|=1}(\eta,p)_n^8d\eta -
q^{-1}\right )
\end{equation}
In the integration domain where $\eta\in 1+{\mathcal P}$, we obtain
\begin{equation}\label{g240}\notag
I_{2422}=-q^{10}\sum_{m=1}^\infty q^{3m}
f_{W_a}(h(p^{m+4},p^{m+4}))\int\limits_{|\epsilon|=1}
(\epsilon,p)_n^{2m+8}\int\limits_{|t|<1}\psi(t^2p^{-m}
\epsilon)dt d\epsilon
\end{equation}
Write $t=p^k\nu$, then shifting $m\to m+4$ and changing variables 
$\epsilon\to \epsilon\nu^{-2}$, we obtain
\begin{equation}\label{g241}\notag
I_{2422}=-q^{10}\sum_{m=5}^\infty q^{3m-12}
f_{W_a}(h(p^m,p^m))\int\limits_{|\nu|=1}(\nu,p)_n^{-4m}d\nu
\int\limits_{|\epsilon|=1}
(\epsilon,p)_n^{2m}\sum_{k=1}^\infty \psi(p^{2k-m+4}
\epsilon) d\epsilon
\end{equation}
Shift $k\to k-2$ and write
\begin{equation}\label{g242} 
-\sum_{m=5}^\infty \sum_{k=3}^\infty = -\sum_{m=1}^\infty \sum_{k=0}^\infty + \sum_{m=1}^\infty \sum_{k=0}^2 + \sum_{m=1}^4 \sum_{k=3}^\infty
\end{equation}
It is not hard to check that the first summand on the right hand side of equation \eqref{g242} cancels with integral $I_{12}$. This can be 
verified without actually computing the terms. 
The third summand on the right hand side of equation \eqref{g242} is 
equal to
\begin{equation}\label{g243} 
\sum_{m=1}^4 q^{3m}\sum_{k=3}^\infty q^{-k}
f_{W_a}(h(p^m,p^m))\int\limits_{|\nu|=1}(\nu,p)_n^{-4m}d\nu
\int\limits_{|\epsilon|=1}
(\epsilon,p)_n^{2m}d\epsilon=
\end{equation}
$$=q^{-3}\sum_{m=1}^4 q^{3m}f_{W_a}(h(p^m,p^m))\int\limits_{|\epsilon|=1}
(\epsilon,p)_n^{2m}d\epsilon$$
In the last equality we used that fact the $f_{W_a}(h(p^m,p^m))=0$
unless $(\eta,p)^{4m}_n=1$. This follows from \eqref{cocycle1}.

The second summand on the right hand side of equation  \eqref{g242}
is equal to 
\begin{equation}\label{g244}\notag 
\sum_{m=1}^\infty q^{3m}
f_{W_a}(h(p^m,p^m))\int\limits_{|\nu|=1}(\nu,p)_n^{-4m}d\nu
\int\limits_{|\epsilon|=1}
(\epsilon,p)_n^{2m}(\psi(p^{-m}\epsilon)+q^{-1}\psi(p^{2-m}\epsilon)
+q^{-2}\psi(p^{4-m}\epsilon))d\epsilon
\end{equation}
Thus, if $m>5$ we get zero contribution. Adding this to integral 
\eqref{g243}, we obtain 
\begin{equation}\label{g245} 
I_{2422}=q^{5/2}f_{W_a}(h(p,p))\int\limits_{|\nu|=1}(\nu,p)_n^{-4}d\nu
+q^{-2}\sum_{m=1}^4 q^{3m}
f_{W_a}(h(p^m,p^m))
\int\limits_{|\epsilon|=1}
(\epsilon,p)_n^{2m}d\epsilon
\end{equation}
$$+q^{-1}\sum_{m=1}^3 q^{3m}
f_{W_a}(h(p^m,p^m))\int\limits_{|\nu|=1}(\nu,p)_n^{-4m}d\nu
\int\limits_{|\epsilon|=1}
(\epsilon,p)_n^{2m}\psi(p^{2-m}\epsilon)d\epsilon+$$
$$+q^{13}f_{W_a}(h(p^5,p^5))\int\limits_{|\nu|=1}(\nu,p)_n^{-20}d\nu
\int\limits_{|\epsilon|=1}
(\epsilon,p)_n^{10}\psi(p^{-1}\epsilon)d\epsilon$$

To complete the proof of the Proposition we need to analyse the 
contribution from $f_{W_a}(e)$, and from the integrals given by
\eqref{g216}, \eqref{g220}, \eqref{g224}, \eqref{g233}, \eqref{g235},
\eqref{g239}, and \eqref{g245}. For that we need to compute 
\begin{equation}\label{g246}\notag
f_{W_a}(h(p^{l+k},p^k))=\int\limits_F f(w_a x_a(y)h(p^{l+k},p^k))
\psi(y)dy
\end{equation}
where we write $f$ for $f^{(r)}$.
Conjugating the torus to the left and then breaking the integration domain in $y$ to $|y|\le 1$ and $|y|>1$, we obtain, 
\begin{equation}\label{g247}\notag
q^{-l}f(h(p^{k},p^{l+k}))+q^{-l}\sum_{m=1}^\infty q^m
f(h(p^{k+m},p^{l+k-m}))\int\limits_{|\epsilon|=1}(\epsilon,p)_n^{6m}
\psi(p^{l-m}\epsilon)d\epsilon
\end{equation}
In the above we also used \eqref{cocycle1} to determine that the 
product $h(p^m,p^{-m})h(\epsilon,\epsilon^{-1})$ contributes the
Hilbert symbol $(\epsilon,p)_n^{6m}$. The contribution to the sum over $m$ is zero if $m>l+1$. Hence, the above is equal to 
\begin{equation}\label{g248}
q^{-l}f(h(p^{k},p^{l+k}))+q^{-l}\sum_{m=1}^l q^m
f(h(p^{k+m},p^{l+k-m}))\int\limits_{|\epsilon|=1}(\epsilon,p)_n^{6m}
d\epsilon + 
\end{equation}
$$+qf(h(p^{k+l+1},p^{k-1}))
\int\limits_{|\epsilon|=1}(\epsilon,p)_n^{6(l+1)}
\psi(p^{-1}\epsilon)d\epsilon$$
Given $l$ and $k$, it is now easy to evaluate \eqref{g248}. For example,
to compute $f_{W_a}(e)$ we need to plug in \eqref{g248} the values $l=k=0$. We then get 
\begin{equation}\label{g249}\notag
f_{W_a}(e)=f(e)+qf(h(p,p^{-1}))\int\limits_{|\epsilon|=1}(\epsilon,p)_n^{6}\psi(p^{-1}\epsilon)d\epsilon
\end{equation}
It follows from \eqref{cocycle1} that $f(h(p,p^{-1}))=0$ unless $n=r=3$. In that case we obtain $f(h(p,p^{-1}))=\chi_1\chi_2^{-1}q^{-1}$. Hence, if $r=3$ then $f_{W_a}(e)=1-\chi_1\chi_2^{-1}q^{-1}$ and for all other values of $r$ we get $f_{W_a}(e)=1$.

As an another example, consider the first summand in \eqref{g216}. In this case $l=2$ and $k=1$, and we obtain
\begin{equation}\label{g250}\notag 
f_{W_a}(h(p^3,p))=q^{-2}f(h(p,p^3))+q^{-2}\sum_{m=1}^2 f(h(p^{m+1},
p^{3-m}))\int\limits_{|\epsilon|=1}(\epsilon,p)_n^{6m}d\epsilon
+
\end{equation}
$$+qf(h(p^4,1))\int\limits_{|\epsilon|=1}(\epsilon,p)_n^{18}\psi(p^{-1}\epsilon)d\epsilon$$
Using \eqref{cocycle1}, the first summand is zero, and also the 
corresponding summand when $m=2$. When $m=1$ the integral over $\epsilon$ implies that that if $r\ne 3$ the corresponding term is zero. However, if $r=3$ then it follows from \eqref{cocycle1} that
$f(h(p^2,p^2))=0$. Hence, the only nonzero contribution is from
the last summand and this happens only if $r=2$ or $r=4$. In this case, the integral over $\epsilon$ is equal to $q^{-1/2}$ if $r=2$,
and if $r=4$ then $n=8$, and the integral  is equal to $q^{-1/2}G_2^{(8)}$. Overall, the first summand of \eqref{g216} is equal to 
$-q^{-2}\chi_1^4$ when $r=2,4$ and zero for all other values of $r$. All other terms are treated in the same way. We also mention that there
are some cancellations between the terms. For example, the term in
equation \eqref{g220} is cancelled with the contribution when $m=4$
in the second summand of \eqref{g245}. We will omit the details.

To complete the proof of the Proposition we still need to prove 
identity \eqref{g234}. A direct calculation is possible, however the following argument is simpler. Let $A$ denote the left hand side of equation
\eqref{g234}. When $r=9$ we get
zero contribution from all terms except from $f_{W_a}(e)$ and  \eqref{g231}. Thus we deduce that 
\begin{equation}\label{g251}
W^{(9)}(e)=1-q^{12}\overline{G_2^{(n)}(1,p)}G_8^{(9)}(1,p)A
f_{W_a}(h(p^5,p^4))
\end{equation}
Here we used the fact that  $f_{W_a}(e)=1$ when $r=9$.
Using equation \eqref{g248} with $k=4$ and $l=1$, we obtain 
$$f_{W_a}(h(p^5,p^4))=q^{1/2}G_3^{(9)}(1,p)f(h(p^6,p^3))=
q^{-29/2}G_3^{(9)}(1,p)\chi_1^6\chi_2^3$$
Hence $W^{(9)}(e)=1-Aq^{-5/2}\overline{G_2^{(n)}(1,p)}
G_8^{(9)}(1,p)G_3^{(9)}(1,p)\chi_1^6\chi_2^3$. From \cite{Gao} Section 7 it
follows that the Theta representation of the group $G_2^{(9)}$ has
no nonzero Whittaker function. It also follows from that reference, see bellow Subsection \ref{prg2}, that the parameters of the Theta
representations are $\chi_1=q^{4/9}$ and $\chi_2=q^{1/9}$. Thus,
plugging this in the above identity, we obtain $0=1-Aq^{1/2}\overline{G_2^{(n)}(1,p)}G_8^{(9)}(1,p)G_3^{(9)}(1,p)$. From this 
equation \eqref{g234} follows.

\end{proof}

Based on the proof of Proposition \ref{prop2}, we can compute the
values of $W^{(r)}_\eta(e)$ for other tori elements $\eta$. Let
$\eta(\alpha,\beta)=\text{diag}(p^\alpha,p^\beta,p^{\alpha-\beta},
1,p^{-\alpha+\beta},p^{-\beta},p^{-\alpha})$. Here $\alpha$ and 
$\beta $ are two integers. We start with
\begin{proposition}\label{prop3}
Integral $W^{(r)}_{\eta(\alpha,-\alpha)}(e)$ converges in the domain $|\chi_1\chi_2|<q^{1/2}$. We have,
\begin{equation}\label{two1}
W^{(2)}_{\eta(1,-1)}(e)=q^{-3/2}\chi_1^2\chi_2^{-2}(1+\chi_2^2)
(1+q^{-1}\chi_1^2\chi_2)(1-q^{-1}\chi_2^2)(1-q^{-1}\chi_1^2\chi_2)
\end{equation}
\begin{equation}\label{four1}
W^{(4)}_{\eta(3,-3)}(e)=q^{-7/2}G_6^{(8)}(1,p)\chi_1^4
\chi_2^{-4}(1-q^{-1}\chi_2^4)
\end{equation}
\begin{equation}\label{six1}
W^{(6)}_{\eta(1,-1)}(e)=q^{-3/2}\chi_1^2\chi_2^{-2}
(1-q^{-2}\chi_1^4\chi_2^2+(1-q^{-1})q^{-1}\chi_1^2\chi_2^4)
\end{equation}
\begin{equation}\label{nine1}
W^{(9)}_{\eta(5,-5)}(e)=q^{-11/2}G_6^{(9)}(1,p)\chi_1^6
\chi_2^{-6}(1-q^{-2}\chi_1^3\chi_2^6)
\end{equation}
\end{proposition} 
\begin{proof}
First, we mention that the cases when $r=3$ and $r=5$ are not included
since they give no new information. To prove the above result we 
claim that all we need to do is to consider the expressions given
by $f_{W_a}(e)$, and from the integrals given by
\eqref{g216}, \eqref{g220}, \eqref{g224}, \eqref{g233}, \eqref{g235},
\eqref{g239}, and \eqref{g245}, where we replace $f_{W_a}(h)$  by
$f_{W_a}^{\alpha}(h)$. Here, $h$ is a diagonal matrix, and
\begin{equation}\label{g260}
f_{W_a}^{\alpha}(h)=\int\limits_F f(h(p^\alpha,p^{-\alpha})w_a
x_a(y)h)\psi(y)dy
\end{equation}
To assert this, we notice that in the case when  $\alpha=0$, 
the only property which we used and is needed to be check here, is the
the property that $f_{W_a}(h(\epsilon,\epsilon)h)=f_{W_a}(h)$ for all $\epsilon$ which is unit. In other words, we need to verify that 
for $|\epsilon|=1$ we have 
$f_{W_a}^\alpha(h(\epsilon,\epsilon)h)=f_{W_a}^\alpha(h)$. But this
easily follows from \eqref{cocycle1}, since the conjugation 
$h(\epsilon^{-1},\epsilon^{-1})h(p^\alpha,p^{-\alpha})h(\epsilon,\epsilon)$ produces no Hilbert symbol. 

To illustrate this computation, consider the case when $r=6$. It 
is not hard to check that we obtain 
$$W^{(6)}_{\eta(\alpha,-\alpha)}(e)=f_{W_a}^\alpha(e)-q^8f_{W_a}^\alpha(h(p^4,p^2))+(1-q^{-1})q^{15/2}f_{W_a}^\alpha(h(p^3,p^3))$$
Indeed, the second summand on the right hand side is derived from the
first summand in equation \eqref{g224}, and the last summand on the the right hand side is derived from the second summand in equation \eqref{g245}. All of the other terms contribute zero when $r=6$.
Using \eqref{g248}, we obtain 
$$f_{W_a}^\alpha(e)=f(h(p^{\alpha},p^{-\alpha})+q^{1/2}
f(h(p^{\alpha+1},p^{-\alpha-1}))$$
$$f_{W_a}^\alpha(h(p^4,p^2))=q^{-2}
f(h(p^{\alpha+2},p^{-\alpha+4}))+(1-q^{-1})
f(h(p^{\alpha+4},p^{-\alpha+2}))+q^{1/2}
f(h(p^{\alpha+5},p^{-\alpha+1}))$$
$$f_{W_a}^\alpha(h(p^3,p^3))=f(h(p^{\alpha+3},p^{-\alpha+3}))
+q^{1/2}f(h(p^{\alpha+4},p^{-\alpha+2}))$$
Plugging $\alpha=1$ we obtain identity \eqref{six1}.

\end{proof}

When $r=2,3$ we can get some more information if we compute some
more Whittaker functions. We have,
\begin{proposition}\label{prop4}
Integrals $W^{(2)}_{(1,0)}(e)$,\  $W^{(2)}_{(0,1)}(e)$ and $W^{(3)}_{(0,2)}(e)$ converges in the domain $|\chi_1\chi_2|<q^{1/2}$. In that domain we have $W^{(2)}_{(1,0)}(e)=0$, and 
\begin{equation}\label{two2}
W^{(2)}_{(0,1)}(e)=q^{-3/2}\chi_2^2(1-q^{-1}\chi_1^2\chi_2^{-2})
(1-q^{-2}\chi_1^4\chi_2^4)
\end{equation}
Also, $W^{(3)}_{(0,2)}(e)$ is equal to
\begin{equation}\label{three2}
q^{-5/2}\chi_2^3(1-q^{-1}\chi_1\chi_2^{-1})
\left [q^{-5/2}\chi_1^5\chi_2^4+G_2^{(3)}(1,p)(1+\chi_1\chi_2^{-1}
-q^{-1}\chi_1^3-q^{-1}\chi_1^2\chi_2-q^{-1}\chi_1\chi_2^2)\right ]
\end{equation}
\end{proposition}
\begin{proof}
To prove this Proposition we cannot directly use the computations carried out in the proof of Proposition \ref{prop2}. This is because we no longer have the analogous  condition to the condition $f_{W_a}(h(\epsilon,\epsilon)h)=f_{W_a}(h)$
which we used there. Nevertheless the computations are similar and will be omitted here.

\end{proof}

\section{ The proof of Theorem \ref{th1}}\label{prth}
In this section we prove Theorem \ref{th1}. For the group $G_2$, in most cases the proof follows immediately from the result of the previous Section. For example, for the group $G_2^{(3)}$ we get the
condition $\chi_1=q\chi_2$ from \eqref{three} and  \eqref{three2}.

 However, for the group $GL_n$ and also for $G_2$ with $r=5$, this is not enough. In these cases the proof is based on the study of intertwining operators. Using similar notations as in \cite{K-P}, we denote
\begin{equation}\label{pr1}\notag
\lambda_\eta(f_\chi^{(r)})=\int\limits_U f^{(r)}_\chi(\eta w_0u)\psi_U(u)du
\end{equation}
Here $f_\chi^{(r)}$ is the unramified vector in $Ind_{B^{(r)}}^{G^{(r)}}\chi\delta_B^{1/2}$, and $\eta\in T_0^{(r)}\backslash T^{(r)}$. When $\eta=1$, we shall denote 
$\lambda_\eta(f_\chi^{(r)})$ by $\lambda(f_\chi^{(r)})$.
Given $w\in W$, let $I_w$ denote the
corresponding intertwining operator, and suppose that for the
specific $\chi$ in question, this intertwining operator $I_wf_\chi^{(r)}$ converges. 
Then we have 
\begin{equation}\label{pr2}
\lambda(I_wf_\chi^{(r)})=c_w(\chi)\lambda(f_{^w\chi}^{(r)})
\end{equation}
Here $c_w(\chi)$ is a scalar depending on $\chi$, and $^w\chi(t)=
\chi(w^{-1}tw)$ for all $t\in T$. 

Fix an unramified character $\chi_0$ of $T$. 
Suppose that the unramified sub-representation  of  $Ind_{B^{(r)}}^{G^{(r)}}\chi_0\delta_B^{1/2}$ has no
nonzero Whittaker function. In particular this means that the meromorphic continuation of
$\lambda(f_\chi^{(r)})$ and $\lambda(I_wf_\chi^{(r)})$ are zero at
$\chi_0$ for all $w$ as above. We mention that to apply the intertwining operator $I_w$ we should check that it is well defined at $\chi=\chi_0$. Depending on $\chi_0$ we then find Weyl
elements such that the above vanishings implies the Theorem. 

\subsection{ The proof in the case $G=GL_n$}\label{prgl}
The Theorem is clearly true if $r\ge n$. So we assume that $r=n-1$. 
It follows from \cite{K-P} that $\Theta^{(n-1)}$ is the unramified 
constituent of $Ind_{B^{(r)}}^{G^{(r)}}\mu\delta_B^{-1/2(n-1)}\delta_B^{1/2}$. Here $\mu$ is a character of $F^*$. It is viewed as a character
of $T^{n-1}$ as follows. For $t=\text{diag}(t_1,\ldots,t_n)\in T^{n-1}$ define $\mu(t)=\mu(t_1)\mu(t_2)\ldots\mu(t_n)$. Thus, to prove the Theorem, we need to prove that under the assumption that the unramified subrepresentation of
$Ind_{B^{(r)}}^{G^{(r)}}\chi_0\delta_B^{1/2}$ has no nonzero Whittaker 
functions, then $\chi_0=\mu\delta_B^{-1/2n}$ for a suitable character
$\mu$. 

For all $1\le i\le n$, denote  $a_i=\chi_{0,i}(p)$. We start by proving
that for all $1\le i\le n-1$ we have $a_i=q^{j/(n-1)}a_n$ for some
$n-i\le j\le n-1$. We argue by induction on $i$. If $i=1$, then
it follows from the assumption and from Proposition \ref{prop1} that
$1-q^{-(n-1)}a_1^{n-1}a_n^{-(n-1)}=0$. Hence $a_1=qa_n$. Consider $i=2$. If 
$a_2=a_1$, there is nothing to prove. If $a_2\ne a_1$, then the
the intertwining operator $I_{w_1}$ is well defined at $\chi_0$, and
hence we obtain from \eqref{pr2} that 
$$(1-q^{-1}a_1^{n-1}a_2^{-(n-1)})(1-q^{-(n-1)}a_2^{n-1}a_n^{-(n-1)})=0$$ 
The second term is zero if $a_2=qa_n=a_1$. This contradicts the
assumption that $a_2\ne a_1$. Thus, the first term must be zero, 
and we get $a_2=q^{-1/(n-1)}a_1=q^{(n-2)/(n-1)}a_n$. To prove the 
above for $a_i$, we may assume that $a_i\ne a_l$ for all $1\le l\le i-1$. Indeed, if $a_i=a_l$ then by the induction hypothesis $a_i$
has the right form.  Thus we may apply the intertwining operator $I_{w_1}\circ I_{w_2}\circ\ldots\circ I_{w_{i-1}}$. At the point $\chi=\chi_0$ we obtain from
\eqref{pr2} and from Proposition \ref{prop1}, that
\begin{equation}\label{van1}\notag
(1-q^{-(n-1)}a_i^{n-1}a_n^{-(n-1)})
\prod_{k=1}^{i-1}(1-q^{-1}a_k^{n-1}a_i^{-(n-1)})=0
\end{equation}
The first term cannot be zero, since if so we would get $a_i=qa_n=a_1$
which contradicts the above assumption. Hence, for some $1\le k\le i-1$, we have $a_i=q^{-1/(n-1)}a_k= q^{(j-1)/(n-1)}a_n$ where $n-i\le j\le n-1$, where the last equation follows from the induction step. This proves that $a_iq^{j/(n-1)}a_n$ for the mentioned above values of $j$.

In particular, we obtain that 
$a_i\ne a_n$ for all $1\le i\le n-1$. 
We claim that using the above we must have $a_i=q^{(n-i)/(n-1)}a_n$. Indeed,
when $i=n-1$, we apply the intertwining operator $I_{w_n}$ and
we get the equality
$$(1-q^{-1}a_{n-1}^{n-1}a_n^{-(n-1)})(1-q^{-(n-1)}a_1^{n-1}a_{n-1}^{-(n-1)})=0$$
If the second term is zero, then $a_{n-1}=q^{-1}a_1=a_n$ which contradicts our above assertion that $a_i\ne a_n$ for all $i$. Thus
the first term is zero, and we get $a_{n-1}=q^{1/(n-1)}a_n$. Assume
by induction that $a_k=q^{(n-k)/(n-1)}a_n$ for all $k\ge i+1$, and
we prove that $a_i=q^{(n-i)/(n-1)}a_n$. From the induction assumption,
we deduce that the intertwining operator $I_{w_{n-1}}\circ I_{w_{n-2}}\circ\ldots\circ I_{w_i}$ is well defined at $\chi=\chi_0$, and hence we obtain that
\begin{equation}\label{van2}\notag
(1-q^{-(n-1)}a_1^{n-1}a_i^{-(n-1)})
\prod_{k=i+1}^{n}(1-q^{-1}a_i^{n-1}a_k^{-(n-1)})=0
\end{equation}
The left most term cannot be zero, and plugging $a_k=q^{(n-k)/(n-1)}a_n$, we obtain that $a_i=q^{(n-k+1)/(n-1)}a_n$ for some $k\ge i+1$. 
However, from the above, we know that $a_i=q^{j/(n-1)}a_n$ for some
$n-i\le j\le n-1$. Hence, $n-i\le n-k+1\le n-1$, which implies that 
$k\le i+1$. Thus we must have $k=i+1$, and the result follows. 

We proved that $a_i=q^{(n-i)/(n-1)}a_n$ for all $1\le i\le n$. This 
means that $\chi_0=\mu\delta_B^{-1/2(n-1)}$ where $\mu=|\cdot|^{1/2}
\chi_{n,0}$, and the Theorem follows.

\subsection{ The proof in the case $G=G_2$}\label{prg2}
We first indicate the parameters of the Theta representation $\Theta^{(r)}$. This can be found in \cite{Gao}. Write $Ind_{B^{(r)}}^{G^{(r)}}\chi\delta_B^{1/2}$ as
$Ind_{B^{(r)}}^{G^{(r)}}(\chi_1,\chi_2)\delta_B^{1/2}$. By that
we mean that if we parameterize  $t=\text{diag}(t_1,t_2,t_1^{-1}t_2,1,t_1t_2^{-1},t_2^{-1},t_1^{-1})$, then $\chi(t)=\chi_1(t_1)\chi_2(t_2)$.

Assume first that
$r$ is not divisible by three. In this case $\Theta^{(r)}$ is the
unramified subrepresentation of $Ind_{B^{(r)}}^{G^{(r)}}\delta_B^{-1/2r}\delta_B^{1/2}$. Then,  if we write the representation as
$Ind_{B^{(r)}}^{G^{(r)}}\chi\delta_B^{1/2}$, we have $\chi_1(p)=q^{2/r}$ and $\chi_2(p)=q^{1/r}$. 

When $r$ is divisible by three, the representation $\Theta^{(r)}$ is the unramified subrepresentation of $Ind_{B^{(r)}}^{G^{(r)}}\mu\delta_B^{1/2}$ where $\mu$ is defined as follows. Write $t\in T$ as above. Then we have $\mu(t)=
|t_1|^{-4/r}|t_2|^{-1/r}$. From this we deduce that if
$\Theta^{(r)}$ is the unramified constituent of $Ind_{B^{(r)}}^{G^{(r)}}\chi\delta_B^{1/2}$ then $\chi_1(p)=q^{4/r}$ and $\chi_2(p)=
q^{1/r}$.

It is simple to check that if we plug these values in equations
\eqref{two} - \eqref{nine}, \eqref{two1} - \eqref{nine1}, \eqref{two2}
and \eqref{three2}, we get zero.

It is not hard to check that if we apply the intertwining operator
$I_{w_a}$ to the unramified vector $f_\chi^{(r)}$  in $Ind_{B^{(r)}}^{G^{(r)}}(\chi_1,\chi_2)\delta_B^{1/2}$, then we obtain $I_{w_a}(
f_\chi^{(r)})=c_{w_a}(\chi)f_{^{w_a}\chi}^{(r)}$ where 
$c_{w_a}(\chi)=(1-q^{-1}\chi_1^{r_1}\chi_2^{-r_1})/(1-\chi_1^{r_1}\chi_2^{-r_1})$ and $f_{^{w_a}\chi}^{(r)}$ is the unramified vector in 
$Ind_{B^{(r)}}^{G^{(r)}}(\chi_2,\chi_1)\delta_B^{1/2}$. Here $r_1=r$
if $r$ is not divisible by three, and $r_1=r/3$ if $r$ is divisible 
by three. We remind the reader that when there is no confusion  we write $\chi_i$ in short for $\chi_i(p)$.
In both cases we can apply $I_{w_a}$ only if $\chi_1\ne \chi_2$. Similarly, when we apply $I_{w_b}$, we obtain 
$c_{w_b}(\chi)=(1-q^{-1}\chi_2^{r})/(1-\chi_2^{r})$ and $f_{^{w_b}\chi}^{(r)}$ is the unramified vector in 
$Ind_{B^{(r)}}^{G^{(r)}}(\chi_1\chi_2,\chi_2^{-1})\delta_B^{1/2}$.
Hence, $I_{w_b}$ is defined only if $\chi_2\ne 1$.

To prove the Theorem, assume that the unramified subrepresentation of $Ind_{B^{(r)}}^{G^{(r)}}\chi_0\delta_B^{1/2}$ has no nonzero Whittaker function. Write $a_1=\chi_{0,1}(p)$ and $a_2=\chi_{0,2}(p)$. 

When $r\ne 2,3,5$, the Theorem follows from  Propositions \ref{prop2},
\ref{prop3} and  \ref{prop4}. We
now prove the Theorem for the case $r=5$. From Proposition \ref{prop2}
we obtained $W^{(5)}(e)=1-q^{-3}\chi_1^5\chi_2^5$. Since we assume that $W^{(5)}(e)=0$ we deduce that $a_1a_2=q^{3/5}$. Assume that $a_2\ne 1$. Then we can apply $I_{w_b}$ on $Ind_{B^{(r)}}^{G^{(r)}}\chi_0\delta_B^{1/2}$, and the vanishing of the
left hand side of \eqref{pr2} implies that $(1-q^{-1}a_2^5)
(1-q^{-3}a_1^5)=0$. The first term in this product is just the
numerator of $c_{w_b}(\chi_0)$ and the second term is identity 
\eqref{five} applied to the induced representation 
$Ind_{B^{(r)}}^{G^{(r)}}(\chi_1\chi_2,\chi_2^{-1})\delta_B^{1/2}$ at
the point $\chi=\chi_0$.
If $1-q^{-3}a_1^5=0$ then the identity $a_1a_2=q^{3/5}$ 
implies that $a_2=1$ which is a contradiction. Hence we must
have $1-q^{-1}a_2^5=0$ which implies that $a_2=q^{1/5}$ and 
hence $a_1=q^{2/5}$. These are exactly the parameters of $\Theta^{(5)}$. If $a_2=1$, then  $I_{w_b}$ is not defined on $Ind_{B^{(r)}}^{G^{(r)}}\chi_0\delta_B^{1/2}$, and
it is not hard to check that $I_{w_a}$ will not give us any new 
information. We apply $I_{w_b}\circ I_{w_a}$. This is well defined
since the assumption $a_2=1$ and the condition $a_1a_2=q^{3/5}$ implies that $a_1\ne a_2$. The intertwining operator 
$I_{w_b}\circ I_{w_a}$ maps $Ind_{B^{(r)}}^{G^{(r)}}(\chi_1,\chi_2)\delta_B^{1/2}$ into $Ind_{B^{(r)}}^{G^{(r)}}(\chi_1\chi_2,\chi_1^{-1})\delta_B^{1/2}$ and is well defined. Indeed, $I_{w_a}$
maps $Ind_{B^{(r)}}^{G^{(r)}}(\chi_1,\chi_2)\delta_B^{1/2}$
into $Ind_{B^{(r)}}^{G^{(r)}}(\chi_2,\chi_1)\delta_B^{1/2}$. Since
$\chi_1\ne 1$ we can further apply $I_{w_b}$.
Equality \eqref{pr2} at the point $\chi=\chi_0$ implies
that $(1-q^{-1}a_1^5a_2^{-5})(1-q^{-1}a_1^5)(1-q^{-3}a_2
^5)=0$. Since we assumed that $a_2=1$ then the condition $a_1a_2=q^{3/5}$ implies $a_1=q^{3/5}$, and plugging these values into the above equality we obtain a contradiction. This completes the proof of the Theorem in
the case when $r=5$.

Finally, when $r=3$, as mentioned above, we deduce from Propositions \ref{prop2} and \ref{prop4} that $\chi=(\chi_1,\chi_1,|\cdot |^{-1})$,
and when $r=2$ we obtain from Section \ref{g2} that $\chi$ is one of the three cases mentioned in the introduction.


\begin{thebibliography}{AAAAA}


\bibitem[B-B-C-F-G]{B-B-C-F-G} Brubaker, Ben; Bump, Daniel; Chinta, Gautam; Friedberg, Solomon; Gunnells, Paul E. Metaplectic ice. Multiple Dirichlet series, L-functions and automorphic forms, 65--92, Progr. Math., 300, Birkhäuser/Springer, New York, 2012. 


\bibitem[B-F-G]{B-F-G} Bump, Daniel; Friedberg, Solomon; Ginzburg, David Small representations for odd orthogonal groups. Int. Math. Res. Not. 2003, no. 25, 1363--1393.



\bibitem[C-S]{C-S} Casselman, W.; Shalika, J.
The unramified principal series of p-adic groups. II. The Whittaker function. Compositio Math. 41 (1980), no. 2, 207--231. 


\bibitem[F-G]{F-G} Friedberg, S., Ginzburg, D. Theta Functions on Covers of Symplectic Groups. arXiv:1601.04970



\bibitem[Gao]{Gao} Gao, F.,	 Distinguished theta representations for Brylinski-Deligne covering groups. arXiv:1602.01880



\bibitem[H-R-T] {H-R-T} Howlett, R. B.; Rylands, L. J.; Taylor, D. E. Matrix generators for exceptional groups of Lie type. J. Symbolic Comput. 31 (2001), no. 4, 429--445.









\bibitem[K-P]{K-P} Kazhdan, D. A.; Patterson, S. J. Metaplectic forms. Inst. Hautes Études Sci. Publ. Math. No. 59 (1984), 35--142.



\bibitem[M]{M} McNamara, Peter J. Metaplectic Whittaker functions and crystal bases. Duke Math. J. 156 (2011), no. 1, 1--31.









\end{thebibliography}
\end{document}